\documentclass{scrartcl}

\usepackage[english]{babel}
\usepackage{amsthm}
\usepackage{amsmath}
\usepackage{amsfonts}
\usepackage{amssymb}
\usepackage{cite}
\usepackage{filecontents}
\usepackage{hyperref}
\usepackage{microtype}

\newtheorem{lemma}{Lemma}
\newtheorem{theorem}[lemma]{Theorem}
\newtheorem{corollary}[lemma]{Corollary}

\DeclareMathOperator{\rk}{rk}
\DeclareMathOperator{\weight}{w}
\DeclareMathOperator{\RREF}{R}
\DeclareMathOperator{\disth}{d_h}
\DeclareMathOperator{\dists}{d_s}
\DeclareMathOperator{\distr}{d_r}
\DeclareMathOperator{\pivot}{p}

\newcommand{\F}{\mathbb{F}}
\newcommand{\sm}[1]{\left(\begin{smallmatrix}#1\end{smallmatrix}\right)}
\newcommand{\gaussmset}[2]{\left[\begin{smallmatrix}{#1}\\{#2}\end{smallmatrix}\right]}
\newcommand{\gaussmnum}[3]{\gaussmset{#1}{#2}_{#3}}

\begin{document}
\title{Generalized linkage construction for constant-dimension codes}

\author{Daniel~Heinlein\thanks{D. Heinlein is with the Department of Communications and Networking, Aalto University, Finland, e-mail: firstname.lastname@aalto.fi.}}

\maketitle

\begin{abstract}
A constant-dimension code (CDC) is a set of subspaces of constant dimension in a common vector space with upper bounded pairwise intersection.
We improve and generalize two constructions for CDCs, the \emph{improved linkage construction} and the \emph{parallel linkage construction}, to the \emph{generalized linkage construction} which in turn yields many improved lower bounds for the cardinalities of CDCs; a quantity not known in general.


\end{abstract}

\section{Introduction}
Let $V \cong \F_q^v$ be a $v$-dimensional vector space over a finite field with $q$ elements $\F_q$.
The set of all subspaces of $V$ forms a metric space  with respect to the so-called subspace-distance $\dists(U,W)=\dim(U+W)-\dim(U \cap W)$ cf.~\cite[Lemma~1]{MR2451015}.
A $(v,N,d;k)_q$ constant-dimension code (CDC) is a set of $k$-subspaces of $V$ of cardinality $N$ such that the subspace-distance of each pair of distinct elements, called codewords, is at least $d$.

Coding in this metric space was motivated by K\"{o}tter and Kschischang in~\cite{MR2451015}.
The main question of subspace coding in the constant-dimension case asks for the maximum cardinality $N$ of a $(v,N,d;k)_q$ code.
This maximum cardinality is denoted as $A_q(v,d;k)$.
The homepage \url{http://subspacecodes.uni-bayreuth.de/}, see also the manual~\cite{HKKW2016Tables}, lists the currently best known lower and upper bounds on $A_q(v,d;k)$ for $q \le 9$, $v \le 19$, all $d$, and all $k$.

Many good lower bounds for CDCs arise from linkage type constructions; Section~\ref{sec:prevlink} provides an overview.
This paper generalizes two already successful constructions, the improved linkage construction (Theorem~\ref{theo:linkageHK}) and the parallel linkage construction (Theorem~\ref{theo:linkageCHWX}), to the so-called generalized linkage construction (Theorem~\ref{theo:generalized_linkage}).

According to the ranking in the homepage, the improved linkage construction is among the best known constructions in $\approx 50.7\%$ of the listed parameters while the parallel linkage construction is among the best known constructions in $\approx 6.3\%$ of the listed parameters.
The generalized linkage construction is among the best known constructions in $\approx 52.5\%$ of the listed parameters.

As these numbers change if new bounds are introduced in the database, especially since most linkage type constructions refer back to smaller CDCs as building blocks, we prove in Lemma~\ref{lem:comparison_linkageCHWK} and Lemma~\ref{lem:comparison_linkageHK} that the generalized linkage construction is strictly better for an infinite family of parameters than the parallel linkage construction and the improved linkage construction, respectively.

The rest of the paper is organized as follows.
In Section~\ref{sec:preliminaries}, we introduce the notation used, in particular $q$-binomial coefficients, rank-metric codes and their sizes, and bounds needed for the comparison of the linkage constructions.
We need rank-metric codes having the additional property that each codeword has an upper bounded rank.

Bounds for the cardinalities of these \emph{rank-restricted rank-metric codes} and special cases are determined in Section~\ref{sec:considerLambda}.

Section~\ref{sec:prevlink} provides an overview over two families of linkage type constructions.
Both families are generalized in a single construction in Section~\ref{sec:genlink}.
For some parameters, the new construction is strictly better, as shown in Section~\ref{sec:comparison}.

\section{Preliminaries}\label{sec:preliminaries}

Throughout the paper, we will use the following notation and facts about $q$-binomial coefficients.
For prime powers $q \ge 2$ and $0 \le n$ we use the $q$-numbers $[n]_q = (q^n-1)/(q-1) = \sum_{i=0}^{n-1} q^i$, the $q$-factorials $[n]_q! = \prod_{i=1}^{n}[i]_q$, and the $q$-binomial coefficients $\gaussmnum{v}{k}{q} = \frac{[v]_q!}{[k]_q![v-k]_q!} = \prod_{i=0}^{k-1} \frac{q^{v}-q^{i}}{q^{k}-q^{i}}$ for $0 \le k \le v$.
An empty sum is defined to be $0$ and an empty product is $1$, so that $[0]_q=0$ and $[0]_q!=1$ in particular.
Note that $[n]_q <q^n$.
The set of $k$-subspaces in $\F_q^v$ is denoted as $\gaussmset{\F_q^v}{k}$ and its cardinality is $\gaussmnum{v}{k}{q}$.

The sizes of $q$-binomial coefficients can be estimated with
\begin{lemma}[{\cite[Lemma~4]{MR2451015}, cf.~\cite[Lemma~8]{ubtepub4049}}]\label{lem:qbinestim}
For all prime powers $q \ge 2$ and $0<k<v$, we have
\begin{align*}
1 < \gaussmnum{v}{k}{q} / q^{k(v-k)} < \left(\prod_{i=1}^{\infty} \left(1-q^{-i} \right) \right)^{-1} < 3.47
.
\end{align*}
\end{lemma}

We use the following well known connection between subspaces and full-rank matrices.
The reduced row echelon form of the matrix $A$ is denoted as $\RREF(A)$.
Then, the bijection between subspaces and their canonical basis in reduced row echelon form, written as rows of a matrix, is
\begin{align*}
\tau : \gaussmset{\F_q^v}{k} \rightarrow \left\{A \in \F_q^{k \times v} \mid \rk A = k \land A = \RREF(A) \right\},
\end{align*}
in particular, $U$ is the row-span of $\tau(U)$ for any subspace $U$.
We omit the dependency of $\tau$ on $q,k,v$ as the context determines them and we extend the codomain of $\tau(\cdot)$ by $\tau^{-1}(A) = \tau^{-1}(\RREF(A))$ for any matrix $A$ of full row-rank.

For a matrix $A$ in reduced row echelon form, $\pivot(A)$ is the binary vector with $\pivot(A)_i=1$ iff column $i$ is a pivot column in the matrix $A$.
We extend the domain of $\pivot(\cdot)$ by $\pivot(B)=\pivot(\RREF(B))$ for any matrix $B$ and $\pivot(U)=\pivot(\tau(U))$ for any subspace $U$.

The horizontal concatenation of matrices or vectors $A,B$ of compatible sizes and ambient fields is denoted as $A \mid B$.

In addition to the subspace-distance
\begin{align*}
&\dists(U,W)
\\=&\dim(U+W)-\dim(U \cap W)
\\=&\dim(U)+\dim(W)-2\dim(U \cap W)
\\=&2\dim(U+W)-\dim(U)-\dim(W)
\\=&2\rk\sm{\tau(U)\\\tau(W)}-\dim(U)-\dim(W)
\stepcounter{equation}\tag{\theequation}\label{math:dsrank}
\end{align*}
for two subspaces $U,W$ of a common vector space, we will also need the Hamming-distance $\disth(u,w)=\#\{i \mid u_i \ne w_i\}$ of two vectors $u,w$ in a common vector space, in particular the weight of $u$ defined as $\weight(u)=\disth(u,0)$, and the rank-distance $\distr(A,B)=\rk(A-B)$ of two matrices $A,B$ of compatible sizes and ambient fields.

It is well known that the Hamming-distance of pivot vectors of subspaces lower bounds their subspace-distance.
\begin{lemma}[{\cite[Lemma~2]{MR2589964}}]\label{lem:dhds}
Let $U$ and $W$ be two subspaces of a common vector space, then
\begin{align*}
\disth(\pivot(U),\pivot(W)) \le \dists(U,W).
\end{align*}
\end{lemma}

We will apply
\begin{align}\label{math:dheq}
\disth(u \mid u',v \mid v') = \disth(u,v)+\disth(u',v')
\end{align}
for any $q$-ary vectors of compatible lengths and the lower bound
\begin{align}\label{math:dhlb}
|\weight(u)-\weight(v)| \le \disth(u,v)
\end{align}
for two binary vectors $u,v$ of equal length, as any pivot vector is a binary vector.

We will also make use of
\begin{align*}
\rk(X) \le &\rk(X \mid Y) \le \rk(X) + \rk(Y) \stepcounter{equation}\tag{\theequation}\label{math:rkineq} \\
&\rk(X + Y) \le \rk(X) + \rk(Y) \stepcounter{equation}\tag{\theequation}\label{math:subaddmat} \\
\rk(X) + \rk(Y) -n \le &\rk(X \!\cdot\! Y) \le \min\{\rk(X),\rk(Y)\} \stepcounter{equation}\tag{\theequation}\label{math:sylvesterandtrivmat}
\end{align*}
for any matrices $X,Y$ of compatible sizes and ambient fields, such that $n$ is the number of columns of $X$.

A rank-metric code (RMC) is a subset of $\F_q^{a \times b}$ of cardinality $N$ such that the rank-distance of each pair of codewords is at least $d$.
An RMC is called linear, if it is a subspace of $\F_q^{a \times b}$.
These parameters are abbreviated as $(a \times b,N,d)_q$.
If in addition the rank of each codeword is at most $u$, we augment the notation to $(a \times b,N,d;u)_q$ and refer to it as rank-restricted RMC (RRMC).
The maximum size of an $(a \times b,N,d;u)_q$ RMC is denoted as $\Lambda(q,a,b,d,u)$.
For a $(k \times n,N,d)_q$ RMC $\mathcal{R}$, the lifted RMC of $\mathcal{R}$ is defined as $\{\tau^{-1}(I \mid R) : R \in \mathcal{R} \}$.
It is a $(k+n,N,2d;k)_q$ CDC.

Delsarte~\cite{MR514618} and Gabidulin~\cite{MR791529} determined the maximum cardinality
\begin{align*}
M(q,a,b,d) = \left\lceil q^{\max\{a,b\}(\min\{a,b\}-d+1)} \right\rceil
\end{align*}
of RMCs for all parameters $q,a,b,d$, and $\min\{a,b\} \le u$ and gave constructions to build bound-achieving RMC codes, so-called maximum rank-distance (MRD) codes.

The theory of Delsarte in~\cite{MR514618} allows to determine the rank-distribution in a linear MRD code.
In his words, a linear MRD code in $\F_q^{a \times b}$ and minimum rank-distance $d$ is equivalent to an $(\min\{a,b\},\max\{a,b\},\min\{a,b\}-d+1,q)$-Singleton system and can be seen as a $(d-1)$-codesign of cardinality $q^{\max\{a,b\}(\min\{a,b\}-d+1)}$, which is a set of bilinear forms $X \subseteq \left\{f : \F_q^{\min\{a,b\}} \times \F_q^{\max\{a,b\}} \to \F_q \mid f \text{ bilinear form} \right\}$ such that $\rk(f-g)>d-1$ for all $f \ne g \in X$.

\begin{theorem}[{\cite[Theorem~5.6]{MR514618}, cf.~\cite[Corollary~26]{de2015rank}}]\label{theo:rankdistributionMRD}
The number of matrices with rank $r$ ($d \le r \le \min\{a,b\}$) in a linear MRD code in $\F_q^{a \times b}$ and minimum rank-distance $d$ is given by
\begin{align*}
&D(q,a,b,d,r)
\\=&\gaussmnum{\min\{a,b\}}{r}{q} \cdot \sum_{i=0}^{r-d}
(-1)^i q^{\binom{i}{2}} \gaussmnum{r}{i}{q} \left( q^{\max\{a,b\}(r-d+1-i)}-1 \right).
\end{align*}
\end{theorem}

We abbreviate
\begin{align*}
\Delta(q,a,b,d,u)
=
1+\sum_{i=d}^{\min\{u,a,b\}} D(q,a,b,d,i)
\end{align*}
which is the size of the largest subset of a linear MRD code in $\F_q^{a \times b}$ and minimum rank-distance $d$ such that the rank of each included matrix is at most $u$.

We deliberately allow matrices with zero rows or zero columns and count sets containing only one such matrix with cardinality one and denote the all-zero matrix with $0$ and the identity matrix with $I$.

Theorem~\ref{theo:rankdistributionMRD} and $\Delta(q,a,b,d,u)$ provide only a construction for $(a \times b,N,d;u)_q$ RMC.
In fact, we have
\begin{align*}
\Delta(q,a,b,d,u) \le \Lambda(q,a,b,d,u) \le M(q,a,b,d)
.
\stepcounter{equation}\tag{\theequation}\label{math:ineqDeltaLambdaM}
\end{align*}

The number of matrices of a given rank in a finite vector space is well known.
\begin{theorem}[{\cite{MR1580299},\cite[Theorem~2]{MR1533848}}]\label{theo:NumberMatricesRankRFiniteField}
The number of matrices in $\F_q^{a \times b}$ of rank $r$ is
\begin{align*}
\prod_{i=0}^{r-1} \frac{(q^{a}-q^{i})(q^{b}-q^{i})}{q^{r}-q^{i}}
=
q^{\binom{r}{2}}(q-1)^r[r]_q!\gaussmnum{a}{r}{q}\gaussmnum{b}{r}{q}
.
\end{align*}
\end{theorem}

Due to the following lemma, we can without loss of generality restrict the parameters of a CDC to $2 \le d/2 \le k \le v/2$, see~\cite[Page~33f.]{ubtepub4049} or~\cite[Page~4822]{MR3988525} for an extensive discussion.
\begin{lemma}[{\cite[Page~3582, Equation~4]{MR2451015}}]\label{lem:AqvdkEqAqvdvmk}
For all $q$ and $2 \le d/2 \le k \le v$, we have
\begin{align*}
A_q(v,d;k) = A_q(v,d;v-k).
\end{align*}
\end{lemma}

We will use the following lower bound by Cossidente and Pavese to compare the generalized linkage construction (Theorem~\ref{theo:generalized_linkage}) to the parallel linkage construction (Theorem~\ref{theo:linkageCHWX_trivialimprovement}) in Lemma~\ref{lem:comparison_linkageCHWK} and to the improved linkage construction (Theorem~\ref{theo:linkageHK}) in Lemma~\ref{lem:comparison_linkageHK}.
\begin{theorem}[{\cite[Theorem~4.3]{MR3759908}}]\label{theo:lbAq844}
\begin{align*}
A_q(8,4;4) \ge q^{12}+q^2(q^2+1)^2(q^2+q+1)+1.
\end{align*}
\end{theorem}

For CDCs having $d=2k$, i.e., the minimum subspace-distance is as large as possible, and the dimension of the ambient vector space is a multiple of the dimension of the codewords, Beutelspacher showed that there are bound-achieving codes.
This setting is often referred to as \emph{spread}.
\begin{theorem}[{\cite{MR404010}}]\label{theo:spread}
For all $q$, $1 \le k$, and $k \mid v$, we have
\begin{align*}
A_q(v,2k;k)=[v]_q/[k]_q
.
\end{align*}
\end{theorem}

An easy to use yet strong upper bound for CDCs is the so-called \emph{Anticode bound}.
\begin{theorem}[{\cite[Theorem~5.2]{MR1984479},~\cite[Theorem~1]{MR2810308}}]\label{theo:anticode}
For all $q$ and $2 \le d/2 \le k \le v$, we have
\begin{align*}
A_q(v,d;k) \le \gaussmnum{v}{k-d/2+1}{q} / \gaussmnum{k}{k-d/2+1}{q}
.
\end{align*}
\end{theorem}

\section{Bounds for $\Lambda$}\label{sec:considerLambda}

Here, we adapt the proof of the upper bound of the size of an MRD code to obtain an upper bound on $\Lambda(q,a,b,d,u)$.

This particular upper bound is the Singleton bound applied to the metric space $(\F_q^{a \times b},\distr)$ via the puncturing operation $g : \F_q^{a \times b} \to \F_q^{a \times (b-1)}$ mapping a matrix to its first $b-1$ columns, i.e., it cuts the last column off.

\begin{theorem}\label{theo:upperboundLambda}
Let $2 \le d \le \min\{a,b\}$ and $0 \le u$ be integers.
Then, we have $\Lambda(q,a,b,d,u) \le \Lambda(q,a,b-1,d-1,u)$ and $\Lambda(q,a,b,d,u) \le$
\begin{align*}
\sum_{r=0}^{\min\{u,\min\{a,b\}-d+1\}}
\!\!\!\!\!
q^{\binom{r}{2}}(q\!-\!1)^r[r]_q!\!\gaussmnum{\min\{a,b\}-d+1}{r}{q}\!\!\gaussmnum{\max\{a,b\}}{r}{q}
\!.
\end{align*}
\end{theorem}
\begin{proof}
Let $b \le a$ without loss of generality, otherwise transpose.
Note that, for matrices $A$ and $B$ of compatible size and ambient field, $\rk(A)-\rk(g(A)) \in \{0,1\}$ (cf. Inequality~\eqref{math:rkineq}), which implies $\distr(A,B)-\distr(g(A),g(B)) \in \{0,1\}$, hence $\distr(A,B)-1 \le \distr(g(A),g(B))$ and $\rk(g(A)) \le \rk(A)$.
So the puncturing operation applied to all elements of an $(a \times b,N,d;u)_q$ RMC yields an $(a \times (b-1),N',d-1;u)_q$ RMC and this is an injective map if $2 \le d$, i.e., $N'=N$.
Applying the puncturing operation $d-1$ times yields an $(a \times (b-d+1),N,1;u)_q$ RMC $\mathcal{R}$.
We have
\begin{align*}
\mathcal{R} \subseteq \left\{ A \in \F_q^{a \times (b-d+1)} \mid \rk A \le u \right\}
\end{align*}
and consequently
\begin{align*}
\Lambda(q,a,b,d,u) \le \#\left\{ A \in \F_q^{a \times (b-d+1)} \mid \rk A \le u \right\}
.
\end{align*}
Then, Theorem~\ref{theo:NumberMatricesRankRFiniteField} allows to determine the cardinality of the right hand side and to complete the proof.
\end{proof}

In addition to the recursion provided in Theorem~\ref{theo:upperboundLambda}, i.e., $\Lambda(q,a,b+1,d+1,u) \le \Lambda(q,a,b,d,u)$, a similar argument shows $\Lambda(q,a+1,b,d+1,u) \le \Lambda(q,a,b,d,u)$ and trivially, we also have $\Lambda(q,a,b,d,u-1) \le \Lambda(q,a,b,d,u)$.

The bound in Theorem~\ref{theo:upperboundLambda} is equivalent to $\Lambda(q,a,b,d,u) \le M(q,a,b,d)$, cf. Inequality~\eqref{math:ineqDeltaLambdaM}, iff $\min\{a,b\} < d+u$ and stronger iff $d+u \le \min\{a,b\}$.

$\Lambda(q,a,b,d,u)$ is precisely the clique number of a graph with vertex set $\left\{A \in \F_q^{a \times b} \mid \rk A \le u \right\}$ and two vertices $A$ and $B$ share an edge iff $\distr(A,B) \ge d$.
The number of vertices can be computed by Theorem~\ref{theo:NumberMatricesRankRFiniteField}.

Using \texttt{GAP}~\cite{GAP4} and \texttt{Cliquer}~\cite{niskanen2003cliquer} we compute
\begin{align*}
&\Lambda(2,2,2,2,1)=3, &&\Lambda(2,3,2,2,1)=3, \\
&\Lambda(2,3,3,2,1)=7, &&\Lambda(2,3,3,2,2)=50, \\
&\Lambda(2,4,4,2,1)=15, &&\Lambda(2,4,4,4,2)=5, \\
&\Lambda(3,2,2,2,1)=4, &&\Lambda(3,3,3,2,1)=13, \text{ and}\\
&\Lambda(3,4,2,2,1)=4.
\end{align*}

The bounds of Inequality~\eqref{math:ineqDeltaLambdaM} are
\begin{align*}
1 \le &\Lambda(q,3,3,2,1) \le q^6, \\
q(q^4+q^3+q^2-q-1) \le &\Lambda(q,3,3,2,2) \le q^6, \text{ and} \\
q(q^7\!+\!q^6\!+\!2q^5\!+\!q^4\!-\!q^2\!-\!2q\!-\!1) \le &\Lambda(q,4,4,2,2) \le q^{12}.
\end{align*}

Theorem~\ref{theo:upperboundLambda} implies
\begin{align*}
&\Lambda(q,3,3,2,1) \le q(q^3+q^2-1), \\
&\Lambda(q,3,3,2,2) \le q^6, \text{ and} \\
&\Lambda(q,4,4,2,2) \le q^3(q^7+q^6+q^5-q^4-q^3-q^2+1).
\end{align*}

We can prove basic structure results for RRMCs.

\begin{lemma}\label{lem:structureRRMC}
Let $\mathcal{C}$ be an $(a \times b,N,d;u)_q$ RMC with $2 \le N$, then $d-u \le \rk(A) \le u$ for each matrix $A$ in $\mathcal{C}$, there is at most one matrix $M$ in $\mathcal{C}$ with $\rk(M) < d/2$, and $d \le \distr(X,Y) \le 2u$ for all $X \ne Y \in \mathcal{C}$.
\end{lemma}
\begin{proof}
By Inequality~\eqref{math:subaddmat}, we have $d \le \rk(A-B) \le \rk(A)+\rk(-B) \le \rk(A)+u \Rightarrow d-u \le \rk(A)$ for $A \ne B \in \mathcal{C}$.

Assume there are two distinct matrices $M$ and $M'$ in $\mathcal{C}$ with $\rk(M) < d/2$ and $\rk(M') < d/2$, then again by Inequality~\eqref{math:subaddmat}, we have $d \le \rk(M-M') \le \rk(M)+\rk(-M') <d/2+d/2$, a contradiction.

Inequality~\eqref{math:subaddmat} shows $\rk(X-Y) \le \rk(X)+\rk(-Y) \le 2u$.
\end{proof}

The next lemma is needed in Theorem~\ref{theo:lambdaequalities} to complete the case of ``$d=2u$''.

\begin{lemma}\label{lem:rkouterproductgeneral}
Let $A_1,B_1 \in \F_q^{a \times u}$ and $A_2,B_2 \in \F_q^{u \times b}$.
Then $\rk(A_1A_2-B_1B_2) = 2u$ iff $\rk\sm{A_1 & B_1}=\rk\sm{A_2 \\ B_2}=2u$.
\end{lemma}
\begin{proof}
We have $A_1A_2-B_1B_2 = \sm{A_1 & B_1}\sm{A_2 \\ -B_2} = M$.
Inequality~\eqref{math:subaddmat} shows $\rk(M) \le \rk(A_1A_2)+\rk(B_1B_2) \le u+u$ and Inequality~\eqref{math:sylvesterandtrivmat} implies
\begin{align*}
&\rk\sm{A_1 & B_1}+\rk\sm{A_2 \\ -B_2}-2u
\\\le& \rk(M) \le \min\left\{\rk\sm{A_1 & B_1},\rk\sm{A_2 \\ -B_2}\right\}.
\end{align*}
If $\rk\sm{A_1 & B_1}=\rk\sm{A_2 \\ B_2}=2u$, then this shows $2u \le \rk(M)$.
If $\rk\sm{A_1 & B_1}<2u$ or $\rk\sm{A_2 \\ B_2}<2u$, then this shows $\rk(M) < 2u$.
\end{proof}

We can settle the size of $\Lambda(q,a,b,d,u)$ for many parameters, in particular for all parameters with $u=1$ or all parameters with $d=2u$.

\begin{theorem}\label{theo:lambdaequalities}
\begin{enumerate}
\item If $2u < d$ or if $\min\{a,b\} < d$, then $\Lambda(q,a,b,d,u) = 1$.
\item If $\min\{a,b\} \le u$, then $\Lambda(q,a,b,d,u) = M(q,a,b,d)$.
\item $\Lambda(q,a,b,1,u) = \sum_{r=0}^{\min\{u,a,b\}} q^{\binom{r}{2}}(q-1)^r[r]_q!\gaussmnum{a}{r}{q}\gaussmnum{b}{r}{q}$.
\item $\Lambda(q,a,b,2u,u) = A_q(\min\{a,b\},2u;u)$.
\end{enumerate}
\end{theorem}
\begin{proof}
\begin{enumerate}
\item
Lemma~\ref{lem:structureRRMC} implies the first case with the fact that $\mathcal{C}=\{0\}$ is an $(a \times b,N,d;u)_q$ RMC of maximum cardinality.
If $\min\{a,b\} < d$, then $d \le \distr(A,B) = \rk(A-B) \le \min\{a,b\}$ implies also that $\mathcal{C}=\{0\}$ is of maximum cardinality.
\item
If $\min\{a,b\} \le u$, then $u$ imposes no restriction and any $(a \times b,N,d;u)_q$ RMC, and the latter is an $(a \times b,N,d)_q$ RMC.
\item
If $d=1$, then the unique maximum cardinality code consists of all matrices of rank at most $u$, so that Theorem~\ref{theo:NumberMatricesRankRFiniteField} completes the proof.
\item
The statement is true for $\min\{a,b\} < 2u$ by (1), so we assume $2u \le \min\{a,b\}$.
Let $\mathcal{C}$ be an $(a \times b,N,2u;u)_q$ RMC of size at least two.
Lemma~\ref{lem:structureRRMC} implies that any matrix in $\mathcal{C}$ has rank $u$ and $\distr(A,B)=2u$ for $A \ne B \in \mathcal{C}$.

Any matrix of rank $u$ in $\mathcal{C}$ is a product $AB$ for $A \in \F_q^{a \times u}$ and $B \in \F_q^{u \times b}$, both of full rank $u$.
Hence, by Lemma~\ref{lem:rkouterproductgeneral}, we have $\distr(A_1B_1,A_2B_2)=2u$ iff $\rk\sm{A_1 & B_1}=\rk\sm{A_2 \\ B_2}=2u$.

Let wlog. $b \le a$, then choose $Y \subseteq \F_q^{u \times b}$ such that $Y$ consists of full rank matrices such that $\rk\sm{B_1 \\ B_2}=2u$ for $B_1 \ne B_2 \in Y$, i.e., $\tau^{-1}(Y)$ is a $(b,\#Y,2u;u)_q$ CDC and the subspace distance is precisely $2u \le \dists(\tau^{-1}(B_1),\tau^{-1}(B_2)) = 2\left(\rk\sm{B_1\\B_2}-u\right) \Leftrightarrow 2u \le \rk\sm{B_1\\B_2}$ by Equation~\eqref{math:dsrank}.
Hence, the maximum cardinality of $Y$ is $A_q(b,2u;u)$.

Choose $X \subseteq \F_q^{a \times u}$ such that $X$ consists of full rank matrices such that $\rk\sm{A_1 & A_2}=2u$ for $A_1 \ne A_2 \in X$, as before $\tau^{-1}(X^T)$ is an $(a,\#X,2u;u)_q$ CDC, where $X^T=\{A^T \mid A \in X\}$.
We choose $X$ of size $A_q(b,2u;u) \le A_q(a,2u;u)$.

Each bijection $f: X \to Y$ gives then an $(a \times b,N,2u;u)_q$ RMC of maximum cardinality $A_q(b,2u;u)$, i.e., $\{xf(x) \mid x \in X\}$.
\end{enumerate}
\end{proof}

Note that the proof of (4) in Theorem~\ref{theo:lambdaequalities} shows first that we have a constant-rank RMC in the case of $d=2u$, so that e.g. \cite[Theorem~2]{MR2798987} could also complete the proof.

Theorem~\ref{theo:lambdaequalities} shows that Inequality\eqref{math:ineqDeltaLambdaM} and Theorem~\ref{theo:upperboundLambda} can be arbitrarily bad, in fact, we have
\begin{align*}
&\Delta(q,a,b,2,1) = 1 \in \Theta(1), \\
&\Lambda(q,a,b,2,1) = [\min\{a,b\}]_q \in \Theta(q^{\min\{a,b\}-1}), \\
&\Lambda(q,a,b,2,1) \le 1+(q-1)[\min\{a,b\}-1]_q[\max\{a,b\}]_q \\
&\quad\quad\quad\quad\quad\quad\,\,\,\in \Theta(q^{a+b-2}), \text{ and} \\
&M(q,a,b,2) = q^{\max\{a,b\}(\min\{a,b\}-1)} \in \Theta(q^{ab-\max\{a,b\}}),
\end{align*}
for $2 \le \min\{a,b\}$, using the Landau-$\Theta$ and Lemma~\ref{lem:qbinestim}.

\begin{corollary}\label{cor:lowerboundlambda}
For $1 \le d$ and $0 \le i \le 2u-d$, we have
\begin{align*}
\Lambda(q,a,b,d,u)
&\ge \Lambda(q,a+i,b+2u-d-i,2u,u)
\\&= A_q(\min\{a+i,b+2u-d-i\},2u;u)
\end{align*}
which yields the strongest bound for
\begin{align*}
i=\max\{0,\min\{2u-d,\lfloor (b-a-d)/2 \rfloor +u\}\}.
\end{align*}
\end{corollary}
\begin{proof}
Let $\mathcal{C}$ be an $((a+i) \times (b+2u-d-i),N,2u;u)_q$ RMC.
We apply the puncturing argument of Theorem~\ref{theo:upperboundLambda} $i$ times to the rows and $2u-d-i$ to the columns of any matrix in $\mathcal{C}$ to obtain an $(a \times b,N,d;u)_q$ RMC $\mathcal{C}'$.
Choosing $\mathcal{C}$ of maximum size shows the first inequality.
Next, (4) in Theorem~\ref{theo:lambdaequalities} shows the equality.
The optimal choice of $i$ follows by $A_q(v,d;k) \le A_q(v+1,d;k)$ and $a+i=b+2u-d-i \Leftrightarrow i=(b-a-d)/2+u$.
\end{proof}

Corollary~\ref{cor:lowerboundlambda} is sometimes stronger than Inequality~\eqref{math:ineqDeltaLambdaM}, e.g. $\Lambda(q,5,5,3,2) \ge 1$ by Inequality~\eqref{math:ineqDeltaLambdaM} but $\Lambda(q,5,5,3,2) \ge A_q(5,4;2) = q^3+1$ by Corollary~\ref{cor:lowerboundlambda}, using $i=0$ and the equality follows from~\cite{MR404010}.

We present another lower bound of $\Lambda(q,a,b,d,u)$ involving CDCs and need therefore the following lemma.

\begin{lemma}\label{lem:generallowerboundofdrbyds}
Let $A_1,B_1 \in \F_q^{a \times u}$ and $A_2,B_2 \in \F_q^{u \times b}$ with $\rk(A_1)=\rk(A_2)=\rk(B_1)=\rk(B_2)=u$.
Then
\begin{align*}
&\dists(\tau^{-1}(A_1^T),\tau^{-1}(B_1^T))+\dists(\tau^{-1}(A_2),\tau^{-1}(B_2))
\\\le&2\distr(A_1A_2,B_1B_2).
\end{align*}
\end{lemma}
\begin{proof}
Using Equation~\eqref{math:dsrank}, we have
\begin{align*}
&d_1
=\dists(\tau^{-1}(A_1^T),\tau^{-1}(B_1^T))
=2\left(\rk\sm{A_1^T\\B_1^T}-u\right)
\\=&2\left(\rk\sm{A_1 & B_1}-u\right)
\Leftrightarrow \rk\sm{A_1 & B_1}=d_1/2+u
\end{align*}
and
\begin{align*}
&d_2
=\dists(\tau^{-1}(A_2),\tau^{-1}(B_2))
=2\left(\rk\sm{A_2\\B_2}-u\right)
\\=&2\left(\rk\sm{A_2\\-B_2}-u\right)
\Leftrightarrow \rk\sm{A_2\\-B_2}=d_2/2+u.
\end{align*}

Then Inequality~\eqref{math:sylvesterandtrivmat} and
\begin{align*}
&2\distr(A_1A_2,B_1B_2)
=2\rk\left(\sm{A_1 & B_1}\sm{A_2 \\ -B_2}\right)
\\\ge&2\left(\rk\left(\sm{A_1 & B_1}\right)+\rk\left(\sm{A_2 \\ -B_2}\right)-2u\right)
\\=&2\left(d_1/2+u+d_2/2+u-2u\right)
=d_1+d_2
\end{align*}
complete the proof.
\end{proof}

\begin{theorem}\label{theo:LambdaLB}
Let $2d \le d_1+d_2$, then $\Lambda(q,a,b,d,u) \ge \min\{A_q(a,d_1;u),A_q(b,d_2;u)\}$.
\end{theorem}
\begin{proof}
Choose $\mathcal{Y}$ as a $(b,N,d_1;u)_q$ CDC and $\mathcal{X}$ as an $(a,N,d_2;u)_q$ CDC with $2d \le d_1+d_2$.

Let $Y \subseteq \F_q^{u \times b}$ such that $Y$ consists of full rank matrices and for each $U \in \mathcal{Y}$ there is exactly one $y \in Y$ with $\tau^{-1}(y)=U$.

Let $X \subseteq \F_q^{a \times u}$ such that $X$ consists of full rank matrices and for each $W \in \mathcal{X}$ there is exactly one $x \in X$ with $\tau^{-1}(x^T)=W$.

In particular, $\#Y=\#X=N$.

Then, each bijection $f: X \to Y$ gives then an $(a \times b,N,d;u)_q$ RMC of cardinality $N$, i.e., $\mathcal{C}=\{xf(x) \mid x \in X\}$, because for $A_1A_2 \ne B_1B_2 \in \mathcal{C}$ such that $A_1,B_1 \in \mathcal{X}$ and $A_2,B_2 \in \mathcal{Y}$, Lemma~\ref{lem:generallowerboundofdrbyds} implies
\begin{align*}
&2d \le d_1+d_2 \le \dists(\tau^{-1}(A_1^T),\tau^{-1}(B_1^T))
\\+&\dists(\tau^{-1}(A_2),\tau^{-1}(B_2)) \le 2\distr(A_1A_2,B_1B_2)
\end{align*}
and each matrix $xf(x) \in \mathcal{C}$ has exactly the rank $u$.
\end{proof}

Since we are actually constructing only a subset of an RRMC which is a constant-rank RMC in Theorem~\ref{theo:LambdaLB}, \cite[Proposition~3]{MR2798987} could also be applied.

Note, that Theorem~\ref{theo:LambdaLB} implies the lower bound of the equality in~(4) of Theorem~\ref{theo:lambdaequalities}, i.e., if $d=2u=d_1=d_2$, then
\begin{align*}
&\Lambda(q,a,b,2u,u)
\\\ge&
\min\{A_q(a,2u;u),A_q(b,2u;u)\}
=
A_q(\min\{a,b\},2u;u).
\end{align*}

\section{Previous linkage constructions}\label{sec:prevlink}

All constructions in this section but Theorem~\ref{theo:linkageK} can be proved by Lemma~\ref{lem:main} which then implies supersets of parameters as the original proofs.
Note, that we also allow $\{0\}$ as RMC.

The original linkage construction was independently discovered by Gluesing-Luerssen and Troha in~\cite{MR3543532} and by Silberstein and Horlemann-Trautmann in~\cite{MR3367813}.

Special cases were already used by Gluesing-Luerssen, Morrison, and Troha in~\cite[Theorem~5.1]{MR3348437} for cyclic orbit codes and by Etzion and Vardy in~\cite[Theorem~11]{MR2810308} for spreads.

\begin{theorem}[{\cite[Theorem~2.3]{MR3543532} and~\cite[Theorem~37 and Corollary~39]{MR3367813}}]\label{theo:linkageGLTSHT}
Let $d/2,k,r,s$ be integers with $2 \le d/2 \le k$, $k \le r$, and $k \le s$.
Let $\mathcal{A}$ be an $(r,\#\mathcal{A},d;k)_q$ CDC and $\mathcal{B}$ be an $(s,\#\mathcal{B},d;k)_q$ CDC.
Let $\mathcal{M}$ be a $(k \times s,\#\mathcal{M},d/2)_q$ RMC.
Then
\begin{align*}
&\{\tau^{-1}(\tau(A) \mid M) : A \in \mathcal{A}, M \in \mathcal{M}\}
\\\cup&\{\tau^{-1}(0 \mid \tau(B)) : B \in \mathcal{B}\}
\end{align*}
is an $(r+s,\#\mathcal{A} \cdot \#\mathcal{M} + \#\mathcal{B},d;k)_q$ CDC.

In particular,
\begin{align*}
A_q(r+s,d;k)
\ge
A_q(r,d;k)
\cdot
M(q,k,s,d/2)
+
A_q(s,d;k).
\end{align*}
\end{theorem}

In~\cite{MR3705116}, Heinlein and Kurz combined Theorem~\ref{theo:linkageGLTSHT} with Lemma~\ref{lem:dhds} to get the following so-called \emph{improved linkage construction}.

\begin{theorem}[{\cite[Theorem~18]{MR3705116}, cf.~\cite[Theorem~136]{ubtepub4049}}]\label{theo:linkageHK}
Let $d/2,k,r,s,t$ be integers with $2 \le d/2 \le k \le (r+s)/2$, $k \le r$, $k \le s+t$, and $0 \le t \le k-d/2$.
Let $\mathcal{A}$ be an $(r,\#\mathcal{A},d;k)_q$ CDC and $\mathcal{B}$ be an $(s+t,\#\mathcal{B},d;k)_q$ CDC.
Let $\mathcal{M}$ be a $(k \times s,\#\mathcal{M},d/2)_q$ RMC.
Then
\begin{align*}
&\{\tau^{-1}(\tau(A) \mid M) : A \in \mathcal{A}, M \in \mathcal{M}\}
\\\cup&\{\tau^{-1}(0 \mid \tau(B)) : B \in \mathcal{B}\}
\end{align*}
is an $(r+s,\#\mathcal{A} \cdot \#\mathcal{M} + \#\mathcal{B},d;k)_q$ CDC.

In particular,
\begin{align*}
&A_q(r+s,d;k)
\\\ge&
A_q(r,d;k)
\cdot
M(q,k,s,d/2)
+
A_q(s+k-d/2,d;k).
\end{align*}
\end{theorem}

Theorem~\ref{theo:linkageHK} was again improved by a generalized extension to any subcode of the form $\{\tau^{-1}(\tau(A) \mid M) : A \in \mathcal{A}, M \in \mathcal{M}\}$ by Kurz in~\cite{kurz2019note}.
Using the notation in~\cite{kurz2019note}, for $0 \le w \le v$ let $B_q(v,w,d;k)$ be the maximum cardinality of a $(v,\#\mathcal{B},d;k)_q$ CDC $\mathcal{B}$ such that there is a $w$-subspace $W$ with $\dim(W \cap B) \ge d/2$ for each $B \in \mathcal{B}$.

Unfortunately, the quantity $B_q(v,w,d;k)$ is not known in general as they generalize the numbers $A_q(v,d;k)$ which are not well understood either, but~\cite{kurz2019note} contains a lower bound:

\begin{theorem}[{\cite[Theorem~3.2, Proposition~4.1, and Theorem~4.2]{kurz2019note}}]\label{theo:linkageK}
Let $d/2,k,r,s$ be integers with $2 \le d/2 \le k \le (r+s)/2$, $k \le r$, and $d/2 \le s$.
\begin{enumerate}
\item
Let $\mathcal{A}$ be an $(r,\#\mathcal{A},d;k)_q$ CDC.
Let $\mathcal{M}$ be a $(k \times s,\#\mathcal{M},d/2)_q$ RMC.
Then the $s$-subspace $W=\tau^{-1}(0 \mid I)$ intersects each codeword in $\{\tau^{-1}(\tau(A) \mid M) : A \in \mathcal{A}, M \in \mathcal{M}\}$ trivial.
Let $\mathcal{B}$ be an $(r+s,\#\mathcal{B},d;k)_q$ CDC such that $\dim(W \cap B) \ge d/2$ for each $B \in \mathcal{B}$.
Then
\begin{align*}
\{\tau^{-1}(\tau(A) \mid M) : A \in \mathcal{A}, M \in \mathcal{M}\}
\cup
\mathcal{B}
\end{align*}
is an $(r+s,\#\mathcal{A} \cdot \#\mathcal{M} + \#\mathcal{B},d;k)_q$ CDC.
\item
\begin{align*}
&A_q(r+s,d;k)
\\\ge&
A_q(r,d;k)
\cdot
M(q,k,s,d/2)
+
B_q(r+s,s,d;k)
\end{align*}
\item
For $4 \le k+1 \le w+2 \le v$:
\begin{align*}
B_q(v,w,2k-2;k) \ge A_q(w,2k-4;k-1)
\end{align*}
\item
For $3 \le k \le s+1$:
\begin{align*}
&A_q(r+s,2k-2;k)
\\\ge&
A_q(r,2k\!-\!2;k)
\!\cdot\!
M(q,k,s,k\!-\!1)
\!+\!
A_q(s,2k\!-\!4;k\!-\!1)
\end{align*}
\end{enumerate}
\end{theorem}

Xu and Chen developed in~\cite{MR3849557} a different direction as they incorporate matrices with lower \emph{and upper} bounded ranks in a construction of CDCs.

\begin{theorem}[{\cite[Theorem~3]{MR3849557}}]\label{theo:linkageXC}
Let $d/2,k$ be integers with $2 \le d/2 \le k$.
Let $\mathcal{M}$ be a $(k \times k,\#\mathcal{M},d/2)_q$ RMC and let $\mathcal{R}$ be a $(k \times k,\#\mathcal{R},d/2;k-d/2)_q$ RMC.
Then
\begin{align*}
\{\tau^{-1}(I \mid M) : M \in \mathcal{M}\}
\cup
\{\tau^{-1}(R \mid I) : R \in \mathcal{R}\}
\end{align*}
is a $(2k,\#\mathcal{M} + \#\mathcal{R},d;k)_q$ CDC.

In particular,
\begin{align*}
A_q(2k,d;k)
\ge
M(q,k,k,d/2)
+
\Lambda(q,k,k,d/2,k-d/2)
\end{align*}
\end{theorem}

Finally, Theorem~\ref{theo:linkageXC} was improved by Chen, He, Weng, and Xu in~\cite{chen2019new} by allowing the dimensions of the ambient spaces to vary.
This is the so-called \emph{parallel linkage construction}.

\begin{theorem}[{\cite[Theorem~3.1]{chen2019new}}]\label{theo:linkageCHWX}
Let $d/2,k,n$ be integers with $2 \le d/2 \le k$ and $0 \le n$.
Let $\mathcal{A}$ be a $(k+n,\#\mathcal{A},d;k)_q$ CDC such that each $A \in \mathcal{A}$ is of the form $\tau(A)=(I \mid A')$, i.e., it is a lifted RMC, and $\mathcal{B}$ be an $(n+k,\#\mathcal{B},d;k)_q$ CDC.
Let $\mathcal{M}$ be a $(k \times k,\#\mathcal{M},d/2)_q$ RMC and $\mathcal{R}$ be a $(k \times k,\#\mathcal{R},d/2;k-d/2)_q$ RMC.
Then
\begin{align*}
&\{\tau^{-1}(\tau(A) \mid M) : A \in \mathcal{A}, M \in \mathcal{M}\}
\\\cup&\{\tau^{-1}(R \mid \tau(B)) : R \in \mathcal{R}, B \in \mathcal{B}\}
\end{align*}
is an $(n+2k,\#\mathcal{A} \cdot \#\mathcal{M} + \#\mathcal{R} \cdot \#\mathcal{B},d;k)_q$ CDC.

In particular,
\begin{align*}
&A_q(n+2k,d;k)
\\\ge&
M(q,k,n,d/2)
\cdot
M(q,k,k,d/2)
\\+&
A_q(n+k,d;k)
\cdot
\Lambda(q,k,k,d/2,k-d/2).
\end{align*}
\end{theorem}

Of course, the concatenation of an RMC with an RMC is again an RMC.
To be more precise, if $\mathcal{M}$ is an $(a \times b,\#\mathcal{M},d)_q$ RMC and $\mathcal{N}$ is an $(a \times c,\#\mathcal{N},d)_q$ RMC, then $\{(M \mid N) : M \in \mathcal{M}, N \in \mathcal{N}\}$ is an $(a \times (b+c),\#\mathcal{M}\cdot\#\mathcal{N},d)_q$ RMC, since Inequality~\eqref{math:rkineq} implies
\begin{align*}
&\rk((M \mid N) - (M' \mid N'))
=
\rk(M-M' \mid N-N')
\\\ge&
\max\{\rk(M-M'),\rk(N-N')\}
\ge
d
\end{align*}
for $M,M' \in \mathcal{M}$ and $N,N' \in \mathcal{N}$ with $(M \mid N) \ne (M' \mid N')$.

Hence, we can improve Theorem~\ref{theo:linkageCHWX} to the following construction.

\begin{theorem}\label{theo:linkageCHWX_trivialimprovement}
Let $d/2,k,s$ be integers with $2 \le d/2 \le k$ and $k \le s$.
Let $\mathcal{A}$ be a $(k+s,\#\mathcal{A},d;k)_q$ CDC such that each $A \in \mathcal{A}$ is of the form $\tau(A)=(I \mid A')$, i.e., it is a lifted RMC, and $\mathcal{B}$ be an $(s,\#\mathcal{B},d;k)_q$ CDC.
Let $\mathcal{R}$ be a $(k \times k,\#\mathcal{R},d/2;k-d/2)_q$ RMC.
Then
\begin{align*}
\mathcal{A}
\cup
\{\tau^{-1}(R \mid \tau(B)) : R \in \mathcal{R}, B \in \mathcal{B}\}
\end{align*}
is a $(k+s,\#\mathcal{A} + \#\mathcal{R} \cdot \#\mathcal{B},d;k)_q$ CDC.

In particular,
\begin{align*}
&A_q(k+s,d;k)
\\\ge&
M(q,k,s,d/2)
+
A_q(s,d;k)
\cdot
\Lambda(q,k,k,d/2,k-d/2).
\end{align*}
\end{theorem}

\begin{lemma}
The bound in Theorem~\ref{theo:linkageCHWX_trivialimprovement} is equivalent to the bound in Theorem~\ref{theo:linkageCHWX} iff $k \le n$ or $n=0$ and stronger iff $0 < n < k$.
\end{lemma}
\begin{proof}
Note that $s=n+k$ and $0 \le n \Leftrightarrow k \le s$.
We have
\begin{align*}
& \text{bound in Theorem~\ref{theo:linkageCHWX_trivialimprovement} } \ge \text{ bound in Theorem~\ref{theo:linkageCHWX}} \\
\Leftrightarrow & M(q,k,s,d/2) \ge M(q,k,n,d/2) \cdot M(q,k,k,d/2) \\
\Leftrightarrow & q^{s(k-d/2+1)} \ge \left\lceil q^{\max\{n,k\}(\min\{n,k\}-d/2+1)} \right\rceil \cdot q^{k(k-d/2+1)} \\
\Leftrightarrow & q^{n(k-d/2+1)} \ge \left\lceil q^{\max\{n,k\}(\min\{n,k\}-d/2+1)} \right\rceil.
\end{align*}
If $k \le n$, the exponent $\max\{n,k\}(\min\{n,k\}-d/2+1) = n(k-d/2+1)$ is at least one and both sides of the inequality coincide.

If $0 \le n < \min\{d/2,k\}$, the exponent $\max\{n,k\}
(\min\{n,k\}-d/2+1) = k(n-d/2+1)$ is at most zero, so the right hand side of the inequality is one, while the left hand side is one iff $n=0$ and else greater than one.

If $d/2 \le n < k$, the exponent $\max\{n,k\}(\min\{n,k\}-d/2+1) = k(n-d/2+1)$ is at least $k$, so we continue:
\begin{align*}
\Leftrightarrow & q^{n(k-d/2+1)} \ge q^{k(n-d/2+1)} \\
\Leftrightarrow & n(k-d/2+1) \ge k(n-d/2+1) \\
\Leftrightarrow & n(-d/2+1) \ge k(-d/2+1) \\
\Leftrightarrow & (n-k)(-d/2+1) \ge 0.
\end{align*}
Due to $n < k$ and $2 \le d/2$, the left hand side is at least one, proving the statement.
\end{proof}

If $\mathcal{A}$ in Theorem~\ref{theo:linkageCHWX_trivialimprovement}, $\mathcal{A}$ and $\mathcal{M}$ in Theorem~\ref{theo:linkageCHWX} or $\mathcal{M}$ in Theorem~\ref{theo:linkageXC} are chosen to be of maximum size, respectively, then they give rise to so-called lifted maximum rank-distance (LMRD) codes and any superset of this particular subcode is upper bounded by more elaborate bounds first proved by Etzion and Silberstein in~\cite[Theorems~10 and~11]{MR3015712} and improved by the author in~\cite[Theorem~1]{MR3988525}, cf.~\cite[Proposition~99]{ubtepub4049}.
To overcome this difficulty, we do not restrict this part of the construction to lifted maximum rank-distance codes in Theorem~\ref{theo:generalized_linkage}.

\section{The generalized linkage construction}\label{sec:genlink}

\begin{lemma}\label{lem:main}
Let $k,r,s$ be positive integers and $A,C \in \F_q^{k \times r}$ and $B,D \in \F_q^{k \times s}$ be matrices such that $\rk(A \mid B) = \rk(C \mid D) = k$. If
\begin{enumerate}
\item $\rk A = \rk C = k$ and $d \le \dists(\tau^{-1}(A),\tau^{-1}(C))$,
\item $A = C$ and $\rk A = k$ and $d/2 \le \distr(B,D)$ or
\item $d/2 \le |\rk A - \rk C|$,
\end{enumerate}
then
\begin{align*}
d
\le&
\dists(\tau^{-1}(A \mid B),\tau^{-1}(C \mid D))
\\=&
\dists(\tau^{-1}(B \mid A),\tau^{-1}(D \mid C))
.
\end{align*}
\end{lemma}
\begin{proof}
For two subspaces $U$ and $W$ of dimension $k$ in a common vector space, we have with Equation~\eqref{math:dsrank}
\begin{align*}
d \le \dists(U,W) \Leftrightarrow \rk\sm{\tau(U)\\\tau(W)} \ge k+d/2.
\end{align*}
Since $\rk\sm{M&N\\O&P}=\rk\sm{N&M\\P&O}=\rk\sm{\RREF(N&M)\\\RREF(P&O)}$ for any matrices $M,N,O,P$ with compatible sizes and ambient fields, we get
\begin{align*}
\dists(\tau^{-1}(A \mid B),\tau^{-1}(C \mid D))
=
\dists(\tau^{-1}(B \mid A),\tau^{-1}(D \mid C)).
\end{align*}

The statement in question is $\rk\sm{A&B\\C&D} \ge k+d/2$.

\begin{enumerate}
\item
Using Inequality~\eqref{math:rkineq}, we obtain $\rk\sm{A&B\\C&D} \ge \rk\sm{A\\C}$ and $\dists(\tau^{-1}(A),\tau^{-1}(C)) \ge d$ is equivalent to $\rk\sm{\RREF(A)\\\RREF(C)} = \rk\sm{A\\C} \ge k+d/2$.

\item
Since $A=C$ of full rank, we have $\rk\sm{A&B\\C&D}=\rk\sm{A&B\\A&D}=\rk\sm{A&B\\0&D-B}=\rk\sm{I&0\\0&D-B}=k+\rk(D-B)$ and the definition of the rank-distance concludes this case.

\item
Here, we use Lemma~\ref{lem:dhds}, Equality~\eqref{math:dheq}, Inequality~\eqref{math:dhlb}, and let $[v]$ ($\{v\}$) denote the first $r$ (last $s$) entries in a vector $v$ so that $v=([v] \mid \{v\})$, $\weight([\pivot(A \mid B)]) = \rk A$, and $\weight([\pivot(C \mid D)]) = \rk C$.
Hence,
\begin{align*}
&\dists(\tau^{-1}(A \mid B),\tau^{-1}(C \mid D))
\\\ge
&\disth(\pivot(\tau^{-1}(A \mid B)),\pivot(\tau^{-1}(C \mid D)))
\\=
&\disth(\pivot(A \mid B),\pivot(C \mid D))
\\=
&\disth([\pivot(A \mid B)],[\pivot(C \mid D)])
\\+&\disth(\{\pivot(A \mid B)\},\{\pivot(C \mid D)\})
\\\ge
&|\weight([\pivot(A \mid B)])-\weight([\pivot(C \mid D)])|
\\+&|\weight(\{\pivot(A \mid B)\})-\weight(\{\pivot(C \mid D)\})|
\\=
&|\rk A - \rk C|+|(k-\rk A) - (k-\rk C)|
\\=
&2|\rk A - \rk C|
\end{align*}
shows that $|\rk A - \rk C| \ge d/2$ implies the minimum distance.
\end{enumerate}
\end{proof}

Lemma~\ref{lem:main} can of course be generalized to at least two blocks and this is used in Theorem~\ref{theo:generalized_linkage_multipleblocks}.

\begin{lemma}\label{lem:main_multipleblocks}
Let $k,m,n_i$ be positive integers and $A_i,B_i \in \F_q^{k \times n_i}$ be matrices such that $\rk(A_1 \mid \ldots \mid A_m) = \rk(B_1 \mid \ldots \mid B_m) = k$ ($1 \le i \le m$). If
\begin{enumerate}
\item $\rk A_i = \rk B_i = k$ and $d \le \dists(\tau^{-1}(A_i),\tau^{-1}(B_i))$,
\item $A_i = B_i$ and $\rk A_i = k$ and $d/2 \le \distr((A_1 \mid \ldots \mid A_{i-1} \mid A_{i+1} \mid \ldots \mid A_m),(B_1 \mid \ldots \mid B_{i-1} \mid B_{i+1} \mid \ldots \mid B_m))$ or
\item $d/2 \le |\rk A_i - \rk B_i|$,
\end{enumerate}
for some $i \in \{1,\ldots,m\}$, then
\begin{align*}
d
\le&
\dists(\tau^{-1}(A_{1} \mid \ldots \mid A_{m}),\tau^{-1}(B_{1} \mid \ldots \mid B_{m}))
.
\end{align*}
\end{lemma}
\begin{proof}
Since
\begin{align*}
\rk\sm{ A_{\pi(1)} \mid \ldots \mid A_{\pi(m)} \\ B_{\pi(1)} \mid \ldots \mid B_{\pi(m)} }
=
\rk\sm{ A_{1} \mid \ldots \mid A_{m} \\ B_{1} \mid \ldots \mid B_{m} }
\end{align*}
for any permutation $\pi$ on $\{1,\ldots,m\}$ we can by Equation~\eqref{math:dsrank} assume that $i=1$.

Then all three statements follow by the corresponding three statements in Lemma~\ref{lem:main} using $r = n_1$, $s = n_2 + \ldots + n_m$, $A = A_1$, $B = (A_{2} \mid \ldots \mid A_{m})$, $C = B_1$, and $D = (B_{2} \mid \ldots \mid B_{m})$.
\end{proof}

We use Lemma~\ref{lem:main} to generalize Theorem~\ref{theo:linkageHK} and Theorem~\ref{theo:linkageCHWX_trivialimprovement} in a single construction.

\begin{theorem}\label{theo:generalized_linkage}
Let $d/2,k,r,s,t$ be integers with $2 \le d/2 \le k \le (r+s)/2$, $k \le r$, $k \le s+t$, and $0 \le t \le k-d/2$.
Let $\mathcal{A}$ be an $(r,\#\mathcal{A},d;k)_q$ CDC and $\mathcal{B}$ be an $(s+t,\#\mathcal{B},d;k)_q$ CDC.
Let $\mathcal{M}$ be a $(k \times s,\#\mathcal{M},d/2)_q$ RMC and $\mathcal{R}$ be a $(k \times (r-t),\#\mathcal{R},d/2;k-d/2-t)_q$ RMC.
Then
\begin{align*}
&\{\tau^{-1}(\tau(A) \mid M) : A \in \mathcal{A}, M \in \mathcal{M}\}
\\\cup&\{\tau^{-1}(R \mid \tau(B)) : R \in \mathcal{R}, B \in \mathcal{B}\}
\end{align*}
is an $(r+s,\#\mathcal{A} \cdot \#\mathcal{M} + \#\mathcal{R} \cdot \#\mathcal{B},d;k)_q$ CDC.

In particular,
\begin{align*}
&A_q(r+s,d;k)
\\\ge&
A_q(r,d;k)
\cdot
M(q,k,s,d/2)
\\+&
A_q(s+t,d;k)
\cdot
\Lambda(q,k,r-t,d/2,k-d/2-t).
\end{align*}
\end{theorem}
\begin{proof}
The distinctness of codewords follows from the minimum distance.

Let $A,A' \in \mathcal{A}$ and $M,M' \in \mathcal{M}$.
If $A=A'$, then~(2) in Lemma~\ref{lem:main} shows $d \le \dists(\tau^{-1}(\tau(A) \mid M),\tau^{-1}(\tau(A') \mid M'))$, else, i.e., $A \ne A'$, then~(1) in Lemma~\ref{lem:main} shows the same statement.

Let $R,R' \in \mathcal{R}$ and $B,B' \in \mathcal{B}$.
If $B=B'$, then~(2) in Lemma~\ref{lem:main} shows $d \le \dists(\tau^{-1}(R \mid \tau(B)),\tau^{-1}(R' \mid \tau(B')))$, else, i.e., $B \ne B'$, then~(1) in Lemma~\ref{lem:main} shows the same statement.

Let $A \in \mathcal{A}$, $M \in \mathcal{M}$, $R \in \mathcal{R}$, and $B \in \mathcal{B}$.
By $[\tau(B)]$ ($\{\tau(B)\}$) we denote the first $t$ (last $s$) columns of $\tau(B)$, in particular $\tau(B) = ([\tau(B)] \mid \{\tau(B)\})$.
Then, the condition in~(3) in Lemma~\ref{lem:main} is $d/2 \le |\rk(\tau(A))-\rk(R \mid [\tau(B)])| = |k-\rk(R \mid [\tau(B)])| = k-\rk(R \mid [\tau(B)]) \Leftrightarrow \rk(R \mid [\tau(B)]) \le k-d/2$.
By the choice of $\mathcal{R}$, we have $\rk R \le k-d/2-t$ and $\rk [\tau(B)] \le t$, so that the statement follows with Inequality~\eqref{math:rkineq}.
\end{proof}

If $\mathcal{A}=\{\tau^{-1}(I)\}$, $r=k$, and $t=0$ is chosen in Theorem~\ref{theo:generalized_linkage}, we get Theorem~\ref{theo:linkageCHWX_trivialimprovement} as special case and if $\mathcal{R}=\{0\}$ and $t=k-d/2$, we get Theorem~\ref{theo:linkageHK} as special case.

This fits also in the framework of Kurz~\cite{kurz2019note}, cf. Theorem~\ref{theo:linkageK}, providing an alternative proof of Theorem~\ref{theo:generalized_linkage}.

\begin{lemma}\label{lem:Bimproved}
Let $d/2,k,r,s,t$ be integers with $2 \le d/2 \le k$, $0 \le r$, $k \le s+t$, and $0 \le t \le \min\{k-d/2,r\}$.
Let $\mathcal{B}$ be an $(s+t,\#\mathcal{B},d;k)_q$ CDC, $\mathcal{R}$ be a $(k \times (r-t),\#\mathcal{R},d/2;k-d/2-t)_q$ RMC, and $W=\tau(0 \mid I)$ of dimension $s$ in $\F_q^{r+s}$.

Then $\dim(W \cap \tau^{-1}(R \mid \tau(B)) ) \ge d/2$ for all $R \in \mathcal{R}$ and $B \in \mathcal{B}$.

In particular,
\begin{align*}
&B_q(r+s,s,d;k)
\\\ge&
A_q(s+t,d;k)
\cdot
\Lambda(q,k,r-t,d/2,k-d/2-t).
\end{align*}
\end{lemma}
\begin{proof}
By $[\tau(B)]$ ($\{\tau(B)\}$) we denote the first $t$ (last $s$) columns of $\tau(B)$, in particular $\tau(B) = ([\tau(B)] \mid \{\tau(B)\})$.
Then we have with Inequality~\eqref{math:rkineq}
\begin{align*}
&\dim(W \cap \tau^{-1}(R \mid \tau(B)) )
\\=& \dim(W) \!+\! \dim(\tau^{-1}(R \mid \tau(B))) \!-\! \dim(W \!+\! \tau^{-1}(R \mid \tau(B)) )
\\=& s + k - \dim(\tau^{-1}(0 \mid 0 \mid I) + \tau^{-1}(R \mid [\tau(B)] \mid \{\tau(B)\}) )
\\=& s + k - \rk\sm{ 0 & 0 & I \\ R & [\tau(B)] & \{\tau(B)\} }
\\=& s + k - \rk\sm{ 0 & 0 & I \\ R & [\tau(B)] & 0 } = k - \rk\sm{ R & [\tau(B)] }
\\\ge& k - \rk R - \rk[\tau(B)] \ge k - (k-d/2-t) - (t) = d/2
.
\end{align*}
\end{proof}

Then~(1) and~(2) in Theorem~\ref{theo:linkageK} together with Lemma~\ref{lem:Bimproved} imply Theorem~\ref{theo:generalized_linkage}.

Independent to this paper, He developed in~\cite{he2019construction} a variation of the generalized linkage construction (Theorem~\ref{theo:generalized_linkage}).
The construction~\cite[Theorem~2]{he2019construction} arises as special case of Theorem~\ref{theo:generalized_linkage} if $t=0$ and $\Lambda(q,k,r,d/2,k-d/2)$ is replaced by $\Delta(q,k,r,d/2,k-d/2)-1$.
In particular, the lower bound provided by Theorem~\ref{theo:generalized_linkage} is strictly better than the lower bound provided by~\cite[Theorem~2]{he2019construction}.
Furthermore, \cite[Corollary~1]{he2019construction} requires $k \ge d$, cf.~\cite[Section~4]{he2019construction}, in contrast to Theorem~\ref{theo:generalized_linkage}.
In~\cite[Section~4]{he2019construction}, He asks for generalizations to $k \le d$ and our generalized linkage construction (Theorem~\ref{theo:generalized_linkage}) provides an answer.

Note, that there are infinite families of parameters showing that the consideration of $t>0$ is justified.
For example, consider $v=7=r+s$, $d=4$, and $k=3$.
Then, the maximum cardinalities of the ingredients of the generalized linkage construction are well known and in the notation of Theorem~\ref{theo:generalized_linkage}:

{
\setlength{\tabcolsep}{5pt}
\begin{tabular}{lll|llll|l}
r&s&t&$\#\mathcal{A}$&$\#\mathcal{M}$&$\#\mathcal{R}$&$\#\mathcal{B}$&$A_q(7,4;3)\ge$\\
\hline
3&4&0&$1$&$q^8$&$[3]_q$&$1$&$q^8\!+\!q^2\!+\!q\!+\!1$\\
3&4&1&$1$&$q^8$&$1$&$q^3\!+\!1$&$q^8\!+\!q^3\!+\!1$\\
4&3&0&$1$&$q^6$&$[3]_q$&$1$&$q^6\!+\!q^2\!+\!q\!+\!1$\\
4&3&1&$1$&$q^6$&$1$&$1$&$q^6\!+\!1$\\
5&2&1&$q^3\!+\!1$&$q^3$&$1$&$1$&$q^6\!+\!q^3\!+\!1$\\
\end{tabular}
}

Here, we use $A_q(5,4;2)=q^3+1$ by~\cite{MR404010}, (1), and (4) of Theorem~\ref{theo:lambdaequalities}.
In particular, the largest code constructed by the generalized linkage construction has cardinality $q^8+q^3+1$, uses $r=3$, $s=4$, and $t=1$, and its cardinality is strictly larger than the codes arising by choosing different parameters.

Independently to this paper, Cossidente, Kurz, Marino, and Pavese developed in~\cite[Lemma~4.1]{coss2019combining} a generalization of \cite[Theorem~4.1]{chen2019new} having the property that neither one of Theorem~\ref{theo:generalized_linkage} and \cite[Lemma~4.1]{coss2019combining} is a special case of the other.
In the following theorem, we generalize both constructions in a single construction.

\begin{theorem}\label{theo:generalized_linkage_multipleblocks}
Let $d/2,k,m,n_i,t_i$ be integers with $2 \le m$, $2 \le d/2 \le k \le (\sum_{i=1}^{m} n_i)/2$, $k \le n_i+t_i$, and $0 \le t_i \le \min\{k-d/2,n_{i-1}\}$ ($2 \le i$), $t_1=0$ ($1 \le i \le m$).
Let
\begin{itemize}
\item $\tau^{-1}(\mathcal{C}_i)$ be $(n_i+t_i,C_i,d;k)_q$ CDCs,
\item $\mathcal{M}_i$ be $(k \times n_i,M_i,d/2)_q$ RMCs,
\item $\mathcal{R}_i$ be $(k \times n_i,R_i,d/2;k-d/2)_q$ RRMCs,
\item $\mathcal{S}_i$ be $(k \times (n_i-t_{i+1}),S_i,d/2;k-d/2-t_{i+1})_q$ RRMCs ($i < m$), and
\item $\mathcal{S}_{m}$ be a set of size $S_m=1$ consisting of an empty matrix
\end{itemize}
for all $i \in \{1,\ldots,m\}$, then
\begin{align*}
\bigcup_{i=1}^{m}
\{ \tau^{-1}( r_1 \mid \ldots \mid r_{i-2} \mid s_{i-1} \mid c_i \mid m_{i+1} \mid \ldots \mid m_{m} )
\\
:
r_j \in \mathcal{R}_j \quad (1 \le j \le i-2),
s_{i-1} \in \mathcal{S}_{i-1}, c_i \in \mathcal{C}_i,
\\
m_w \in \mathcal{M}_w \quad (i+1 \le w \le m)
\}
\end{align*}
is a $(\sum_{i=1}^{m} n_i,N,d;k)_q$ CDC with
\begin{align*}
N = \sum_{i=1}^{m} C_i \cdot S_i \cdot  \prod_{j=1}^{i-2} R_j \cdot \prod_{w=i+1}^{m} M_w
.
\end{align*}

In particular,
\begin{align*}
&A_q(\textstyle{\sum_{i=1}^{m} n_i},d;k) \ge A_q(n_1,d;k)
\\&\cdot \prod_{w=2}^{m} M(q,k,n_w,d/2)+\sum_{i=2}^{m} A_q(n_i+t_i,d;k)
\\&\cdot \Lambda(q,k,n_{i-1}-t_{i},d/2,k-d/2-t_{i}) \cdot
\\&\prod_{j=1}^{i-2} \Lambda(q,k,n_j,d/2,k-d/2) \cdot \prod_{w=i+1}^{m} M(q,k,n_w,d/2).
\end{align*}
\end{theorem}
\begin{proof}
The distinctness of codewords follows from the minimum distance.

Let $\tau^{-1}( r_1 \mid \ldots \mid r_{i-2} \mid s_{i-1} \mid c_i \mid m_{i+1} \mid \ldots \mid m_{m} )$ and $\tau^{-1}( r_1' \mid \ldots \mid r_{i-2}' \mid s_{i-1}' \mid c_i' \mid m_{i+1}' \mid \ldots \mid m_{m}' )$ be two distinct codewords in the same subcode.

If $c_i = c_i'$, then (1) in Lemma~\ref{lem:main_multipleblocks} implies the minimum distance using the minimum subspace distance of $\mathcal{C}_i$.
Else, by distinctness, there is a $1 \le j \le i-2$ with $r_j \ne r_j'$ or $s_{i-1} \ne s_{i-1}'$ or there is a $i+1 \le w \le m$ with $m_w \ne m_w'$.
We abbreviate all cases in $x \ne x'$ for $x \in \{r_j,s_{i-1},m_w\}$.
Then, by the minimum rank distance and Inequality~\eqref{math:rkineq}, we have $d/2 \le \distr(x,x') \le \distr( (r_1 \mid \ldots \mid r_{i-2} \mid s_{i-1} \mid m_{i+1} \mid \ldots \mid m_{m}), (r_1' \mid \ldots \mid r_{i-2}' \mid s_{i-1}' \mid m_{i+1}' \mid \ldots \mid m_{m}') )$ so that (2) in Lemma~\ref{lem:main_multipleblocks} concludes this case.

Let $\tau^{-1}( r_1 \mid \ldots \mid r_{i-2} \mid s_{i-1} \mid c_{i} \mid m_{i+1} \mid \ldots \mid m_{m} )$ and $\tau^{-1}( r_1' \mid \ldots \mid r_{i'-2}' \mid s_{i'-1}' \mid c_{i'}' \mid m_{i'+1}' \mid \ldots \mid m_{m}' )$ be two distinct codewords in different subcodes corresponding to $i$ and $i'$, we use without loss of generality $i < i'$.

Define $x$ as $(s_{i'-1}' \mid [c_{i'}'])$ if $i+1 = i'$, where $[c_{i'}']$ is the matrix consisting of the leftmost $t_i$ columns of $c_{i'}'$, and as $r_i'$ if $i+2 \le i'$.
In the first case, we have $\rk x \le \rk(s_{i'-1}') + \rk([c_{i'}']) \le (k-d/2-t_i) + (t_i)$ by Inequality~\eqref{math:rkineq}, so that $\rk x \le k-d/2$ in both cases.

Then we have $|\rk c_i - \rk x| = \rk c_i - \rk x \ge k - (k-d/2) = d/2$ and (3) in Lemma~\ref{lem:main_multipleblocks} concludes this case.
\end{proof}

This relates to other constructions as follows.
If we set $m=2$, we obtain Theorem~\ref{theo:generalized_linkage}.
Using $t=0$, we get \cite[Lemma~4.1]{coss2019combining}.
If $m=s+1$, $n_i=n$, $t_i=0$, $d=2(n-t)$, $k=n$, $\mathcal{C}_i = \{\tau^{-1}(I)\}$, we get \cite[Theorem~4.1]{chen2019new}.

\section{New CDCs and better lower bounds}\label{sec:comparison}

According to the numerical evidence of \url{http://subspacecodes.uni-bayreuth.de/}, the generalized linkage construction (Theorem~\ref{theo:generalized_linkage}) increases all known lower bounds on $A_q(v,4;4)$ for all listed parameters $q$ and $12 \le v$.

In the setting of $q=2$, the previously best known lower bound $A_2(12,4;4) \ge 19\,664\,917$ is given by the improved linkage construction (Theorem~\ref{theo:linkageHK}), our new generalized linkage construction (Theorem~\ref{theo:generalized_linkage}) increases the bound to $A_2(12,4;4) \ge 19\,673\,821$, while the parallel linkage construction (Theorem~\ref{theo:linkageCHWX}) only creates codes of size $19\,297\,741$.

Hence, we compare the sizes of CDCs with $v=12$ and $d=k=4$ constructed by these three constructions for all $q$.

The size of the code constructed in Theorem~\ref{theo:generalized_linkage} using $r=8$ and $t=0$, so that $s=4$, is
\begin{align*}
&A_q(r,d;k) \cdot M(q,k,s,d/2)
\\+&A_q(s+t,d;k) \cdot \Lambda(q,k,r-t,d/2,k-d/2-t)
\\\ge&A_q(8,4;4) \cdot M(q,4,4,2) + A_q(4,4;4) \cdot \Delta(q,4,8,2,2)
\\=&A_q(8,4;4) \cdot q^{4(4-2+1)} + 1 \cdot (1+\gaussmnum{4}{2}{q}(q^8-1))
\\=&A_q(8,4;4) \cdot q^{12} + 1+\gaussmnum{4}{2}{q}(q^8-1) \stepcounter{equation}\tag{\theequation}\label{math:eqGLq844}
\\>&A_q(8,4;4) \cdot q^{12} +\gaussmnum{4}{2}{q}(q^8-1) \stepcounter{equation}\tag{\theequation}\label{math:eqm1GLq844}
\\>&q^{12} \cdot q^{12} + q^4(q^8-1)
\\=&q^{24}+q^{12}-q^4 \stepcounter{equation}\tag{\theequation}\label{math:lbGLq844}
\end{align*}
For the last inequality, we use Theorem~\ref{theo:lbAq844} and Lemma~\ref{lem:qbinestim}.

We compare the generalized linkage construction to the parallel linkage construction.
Unfortunately, the bound of Theorem~\ref{theo:upperboundLambda}, i.e., $\Lambda(q,4,4,2,2) \le q^3(q^7+q^6+q^5-q^4-q^3-q^2+1)$, is too weak to show this result in general.

\begin{lemma}\label{lem:comparison_linkageCHWK}
If $2 \le q$ is a prime power, $v=12$, and $d=k=4$, then Theorem~\ref{theo:generalized_linkage} constructs a larger code than Theorem~\ref{theo:linkageCHWX_trivialimprovement}, utilizing $\Delta$ instead of $\Lambda$ in the latter.
\end{lemma}
\begin{proof}
The size of the code constructed in Theorem~\ref{theo:linkageCHWX_trivialimprovement} uses $s=8$ and is, with $\Lambda$ replaced by $\Delta$,
\begin{align*}
&M(q,k,s,d/2) + A_q(s,d;k) \cdot \Delta(q,k,k,d/2,k-d/2)
\\=&M(q,4,8,2) + A_q(8,4;4) \cdot \Delta(q,4,4,2,2)
\\=&q^{8(4-2+1)} + A_q(8,4;4) \cdot (1+\gaussmnum{4}{2}{q}(q^4-1))
\end{align*}

Due to Equation~\eqref{math:eqGLq844}, we have
\begin{align*}
& q^{24} + A_q(8,4;4) (1+\gaussmnum{4}{2}{q}(q^4-1))
\\<& A_q(8,4;4) q^{12} + 1 + \gaussmnum{4}{2}{q}(q^8-1)
\\\Leftrightarrow& q^{24} \!-\! 1\!-\!\gaussmnum{4}{2}{q}(q^8\!-\!1) < A_q(8,4;4) ( q^{12} \!-\! 1\!-\!\gaussmnum{4}{2}{q}(q^4\!-\!1) )
.
\end{align*}

Using the lower bound of $A_q(8,4;4)$ in Theorem~\ref{theo:lbAq844} we will prove
\begin{align*}
&\frac{q^{24} - 1 - \gaussmnum{4}{2}{q}(q^8-1)}{q^{12} - 1 - \gaussmnum{4}{2}{q}(q^4-1)} < q^{12}\!+\!q^2(q^2\!+\!1)^2(q^2\!+\!q\!+\!1)\!+\!1
\\\Leftrightarrow&\frac{q^{24} - 1 - (q^2+1)(q^2+q+1)(q^8-1)}{q^{12} - 1 - (q^2+1)(q^2+q+1)(q^4-1)}
\\<&q^{12}+q^2(q^2+1)^2(q^2+q+1)+1
\\\Leftarrow&\frac{q^{24} - q^2 q^2(q^8-1)}{q^{12} - 1 - (q^2+1)(q^2+q+1)q^4}
\\\le& q^{12}+q^2(q^2+1)^2(q^2+q+1)
\end{align*}
which is equivalent to the nonnegativity of
\begin{align*}
&(q^{12}+q^2(q^2+1)^2(q^2+q+1))
\\\cdot&(q^{12} \!-\! 1 \!-\! (q^2+1)(q^2+q+1)q^4) \!-\!(q^{24} \!-\! q^2 q^2(q^8\!-\!1))
\\=&q^{2}(q^{16}+q^{15}+q^{14}-q^{13}-5q^{12}-8q^{11}-13q^{10}-12q^{9}
\\-&13q^{8}-8q^{7}-7q^{6}-3q^{5}-4q^{4}-2q^{3}-4q^{2}-q-1)
\\\ge& q^2 \left(q^{16}+q^{15}+q^{14}-q^{13}-13\sum_{i=0}^{12} q^i \right)
\\\ge& q^2(q^{16}+q^{15}+q^{14}-q^{13}-13q^{13})
\\=& q^{15}(q-2)(q^{2}+3q+7)
.
\end{align*}
Since the last term is nonnegative for $2 \le q$, the statement follows.
\end{proof}

We compare the generalized linkage construction to the improved linkage construction.

\begin{lemma}\label{lem:comparison_linkageHK}
If $2 \le q$ is a prime power, $v=12$, and $d=k=4$, then Theorem~\ref{theo:generalized_linkage} constructs a larger code than Theorem~\ref{theo:linkageHK}.
\end{lemma}
\begin{proof}

The size of the code constructed in Theorem~\ref{theo:linkageHK} using $t=2$, $s=12-r$, and $4 \le r \le 10$ is
\begin{align*}
&A_q(r,d;k) \cdot M(q,k,s,d/2) + A_q(s+k-d/2,d;k)
\\=&A_q(r,4;4) \cdot M(q,4,12-r,2) + A_q(14-r,d;4)
\\=&A_q(r,4;4) \cdot q^{\max\{12-r,4\}(\min\{12-r,4\}-1)} + A_q(14-r,4;4)
\\=&
\begin{cases}
A_q(r,4;4) \cdot q^{3(12-r)} + A_q(14-r,4;4) & \text{if } 4 \le r \le 8 \\
A_q(r,4;4) \cdot q^{4(11-r)} + A_q(14-r,4;4) & \text{if } 9 \le r \le 10 \\
\end{cases}
\end{align*}

Using Lemma~\ref{lem:AqvdkEqAqvdvmk} and Theorem~\ref{theo:spread}, we have $A_q(4,4;4)=A_q(4,4;0)=1$, $A_q(5,4;4)=A_q(5,4;1)=1$, $A_q(6,4;4)=A_q(6,4;2)=[6]_q/[2]_q=q^4+q^2+1$, and $A_q(7,4;4)=A_q(7,4;3)$.

For $4 \le x$, the Anticode bound in Theorem~\ref{theo:anticode} and Lemma~\ref{lem:qbinestim} imply $A_q(x,4;4) \le \gaussmnum{x}{4-2+1}{q}/\gaussmnum{4}{4-2+1}{q} < 4q^{3(x-3)}/q^3 = 4q^{3x-12}$.

For $r=4$, Theorem~\ref{theo:linkageHK} yields a CDC of size $q^{24} + A_q(10,4;4)$.
Then, comparing to Inequality~\eqref{math:eqm1GLq844}, i.e.,
\begin{align*}
&A_q(8,4;4) \cdot q^{12} + \gaussmnum{4}{2}{q}(q^8-1) - q^{24} - A_q(10,4;4)
\\\ge&(q^{12}+q^2(q^2+1)^2(q^2+q+1)+1) \cdot q^{12}
\\+& (q^8-1)[4]_q[3]_q/[2]_q - q^{24} - [10]_q[9]_q[8]_q/([4]_q[3]_q[2]_q)
\\=&q^{20}+q^{19}+2q^{18}+2q^{17}+2q^{16}-q^{14}-q^{13}-q^{12}-q^{11}
\\-&q^{10}-q^{9}-2q^{8}-2q^{7}-3q^{6}-q^{5}-3q^{4}-2q^{3}-3q^{2}-q-2
\\>&q^{20}-3\sum_{i=0}^{14}q^i
=q^{20}-3[15]_q
>q^{20}-3q^{15}
>0
,
\end{align*}
concludes this case.

For $r=5$, Theorem~\ref{theo:linkageHK} yields a CDC of size $q^{21} + A_q(9,4;4) \le q^{21} + 4q^{15} \le q^{21} + q^{17} \le q^{22}$, so that Inequality~\eqref{math:lbGLq844} concludes this case.

For $r=6$, Theorem~\ref{theo:linkageHK} yields a CDC of size $([6]_q/[2]_q) q^{18} + A_q(8,4;4) \le 2q^{22} + 4q^{12} \le q^{23} + q^{14} \le q^{24}$, so that Inequality~\eqref{math:lbGLq844} concludes this case.

For $r=7$, Theorem~\ref{theo:linkageHK} yields a CDC of size $A_q(7,4;3) (q^{15}+1)$.
Then, comparing to Inequality~\eqref{math:eqm1GLq844}, i.e.,
\begin{align*}
&A_q(8,4;4) q^{12} + \gaussmnum{4}{2}{q}(q^8-1) - A_q(7,4;3) (q^{15}+1)
\\\ge&q^{24} + (q^8-1)[4]_q[3]_q/[2]_q - (q^{15}+1)[7]_q[6]_q/([3]_q[2]_q)
\\=&q^{24}-q^{23}-q^{21}-q^{20}-q^{19}-q^{18}-q^{17}-q^{15}+q^{12}+q^{11}
\\+&2q^{10}+q^{9}-q^{6}-q^{5}-2q^{4}-2q^{3}-3q^{2}-q-2
\\>&q^{24}-q^{23}-q^{21}-3\sum_{i=0}^{20} q^i
=q^{24}-q^{23}-q^{21}-3[21]_q
\\>&q^{24}-q^{23}-4q^{21}
=(q-2)(q^2+q+2)q^{21}
\ge 0
,
\end{align*}
concludes this case.

For $r=8$, Theorem~\ref{theo:linkageHK} yields a CDC of size $A_q(8,4;4) q^{12} + [6]_q/[2]_q$.
Then, comparing to Inequality~\eqref{math:eqm1GLq844}, i.e.,
\begin{align*}
&A_q(8,4;4) q^{12} + [6]_q/[2]_q \le A_q(8,4;4) \cdot q^{12} + \gaussmnum{4}{2}{q}(q^8-1)
\\\Leftrightarrow&[6]_q/[2]_q \le (q^8-1)[4]_q[3]_q/[2]_q
\\\Leftrightarrow&[6]_q \le (q^8-1)[4]_q[3]_q
\\\Leftarrow&q^6-1 < q^8-1
,
\end{align*}
concludes this case.

For $r=9$, Theorem~\ref{theo:linkageHK} yields a CDC of size $A_q(9,4;4) q^{8} + 1 \le 4q^{23} + 1$, so that Inequality~\eqref{math:lbGLq844} concludes this case for all $4 \le q$.
Next, we use $(\prod_{i=1}^{\infty}(1-3^{-i}))^{-1} < 1.8$ in Lemma~\ref{lem:qbinestim}, so that $A_q(9,4;4) q^{8} + 1 \le 1.8q^{23} + 1$ and Inequality~\eqref{math:lbGLq844} concludes this case for $q=3$, too.
Last, the Anticode bound is precisely $A_2(9,4;4) \le 52\,535$, so that $A_2(9,4;4) 2^{8} + 1 \le 13\,448\,961 < 16\,781\,296 = 2^{24} + 2^{12} -2^4$, concluding the case $r=9$ for all $q$.

For $r=10$, Theorem~\ref{theo:linkageHK} yields a CDC of size $A_q(10,4;4) q^{4} + 1 \le 4q^{22} + 1 \le q^{24} + 1$, so that Inequality~\eqref{math:lbGLq844} concludes this case.

\end{proof}


\begin{thebibliography}{10}

\bibitem{MR404010}
A.~Beutelspacher.
\newblock Partial spreads in finite projective spaces and partial designs.
\newblock {\em Math. Z.}, 145(3):211--229, 1975.

\bibitem{chen2019new}
H.~Chen, X.~He, J.~Weng, and L.~Xu.
\newblock New constructions of subspace codes using subsets of {MRD} codes in
  several blocks.
\newblock {\em arXiv:1908.03804}, 2019.

\bibitem{coss2019combining}
A.~Cossidente, S.~Kurz, G.~Marino, and F.~Pavese.
\newblock Combining subspace codes.
\newblock {\em arXiv:1911.03387}, 2019.

\bibitem{MR3759908}
A.~Cossidente and F.~Pavese.
\newblock Subspace codes in {$\mathrm{PG}(2N-1, Q)$}.
\newblock {\em Combinatorica}, 37(6):1073--1095, 2017.

\bibitem{de2015rank}
J.~de~la Cruz, E.~Gorla, H.~H. L\'{o}pez, and A.~Ravagnani.
\newblock Rank distribution of {D}elsarte codes.
\newblock {\em arXiv:1510.01008}, 2015.

\bibitem{MR514618}
P.~Delsarte.
\newblock Bilinear forms over a finite field, with applications to coding
  theory.
\newblock {\em J. Combin. Theory Ser. A}, 25(3):226--241, 1978.

\bibitem{MR2589964}
T.~Etzion and N.~Silberstein.
\newblock Error-correcting codes in projective spaces via rank-metric codes and
  {F}errers diagrams.
\newblock {\em IEEE Trans. Inform. Theory}, 55(7):2909--2919, 2009.

\bibitem{MR3015712}
T.~Etzion and N.~Silberstein.
\newblock Codes and designs related to lifted {MRD} codes.
\newblock {\em IEEE Trans. Inform. Theory}, 59(2):1004--1017, 2013.

\bibitem{MR2810308}
T.~Etzion and A.~Vardy.
\newblock Error-correcting codes in projective space.
\newblock {\em IEEE Trans. Inform. Theory}, 57(2):1165--1173, 2011.

\bibitem{MR1533848}
S.~D. Fisher and M.~N. Alexander.
\newblock Classroom {N}otes: {M}atrices over a {F}inite {F}ield.
\newblock {\em Amer. Math. Monthly}, 73(6):639--641, 1966.

\bibitem{MR791529}
E.~M. Gabidulin.
\newblock Theory of codes with maximum rank distance.
\newblock {\em Problemy Peredachi Informatsii}, 21(1):3--16, 1985.

\bibitem{MR2798987}
M.~Gadouleau and Z.~Yan.
\newblock Constant-rank codes and their connection to
  constant-dimension codes.
\newblock {\em IEEE Trans. Inform. Theory}, 56(7):3207--3216, 2010.

\bibitem{GAP4}
The GAP~Group.
\newblock {\em {GAP -- Groups, Algorithms, and Programming, Version 4.10.0}},
  2019.

\bibitem{MR3348437}
H.~Gluesing-Luerssen, K.~Morrison, and C.~Troha.
\newblock Cyclic orbit codes and stabilizer subfields.
\newblock {\em Adv. Math. Commun.}, 9(2):177--197, 2015.

\bibitem{MR3543532}
H.~Gluesing-Luerssen and C.~Troha.
\newblock Construction of subspace codes through linkage.
\newblock {\em Adv. Math. Commun.}, 10(3):525--540, 2016.

\bibitem{he2019construction}
X.~He.
\newblock Construction of Const Dimension Code from Two Parallel
  Versions of Linkage Construction.
\newblock {\em arXiv:1910.04472}, 2019.

\bibitem{ubtepub4049}
D.~Heinlein.
\newblock {\em Integer linear programming techniques for constant dimension
  codes and related structures}.
\newblock PhD thesis, Bayreuth, November 2018.

\bibitem{MR3988525}
D.~Heinlein.
\newblock New {LMRD} code bounds for constant dimension codes and improved
  constructions.
\newblock {\em IEEE Trans. Inform. Theory}, 65(8):4822--4830, 2019.

\bibitem{HKKW2016Tables}
D.~Heinlein, M.~Kiermaier, S.~Kurz, and A.~Wassermann.
\newblock Tables of subspace codes.
\newblock {\em arXiv:1601.02864}, 2016.

\bibitem{MR3705116}
D.~Heinlein and S.~Kurz.
\newblock Asymptotic bounds for the sizes of constant dimension codes and an
  improved lower bound.
\newblock In {\em Coding theory and applications}, volume 10495 of {\em Lecture
  Notes in Comput. Sci.}, pages 163--191. Springer, Cham, 2017.

\bibitem{MR2451015}
R.~K\"{o}tter and F.~R. Kschischang.
\newblock Coding for errors and erasures in random network coding.
\newblock {\em IEEE Trans. Inform. Theory}, 54(8):3579--3591, 2008.

\bibitem{kurz2019note}
S.~Kurz.
\newblock A note on the linkage construction for constant dimension codes.
\newblock {\em arXiv:1906.09780}, 2019.

\bibitem{MR1580299}
G.~Landsberg.
\newblock Ueber eine {A}nzahlbestimmung und eine damit zusammenh\"{a}ngende
  {R}eihe.
\newblock {\em J. Reine Angew. Math.}, 111:87--88, 1893.

\bibitem{niskanen2003cliquer}
S.~Niskanen and P.~R.~J. {\"O}sterg{\aa}rd.
\newblock {\em Cliquer User's Guide: Version 1.0}.
\newblock Helsinki University of Technology Helsinki, Finland, 2003.

\bibitem{MR3367813}
N.~Silberstein and A.-L. Trautmann.
\newblock Subspace codes based on graph matchings, {F}errers diagrams, and
  pending blocks.
\newblock {\em IEEE Trans. Inform. Theory}, 61(7):3937--3953, 2015.

\bibitem{MR1984479}
H.~Wang, C.~Xing, and R.~Safavi-Naini.
\newblock Linear authentication codes: bounds and constructions.
\newblock {\em IEEE Trans. Inform. Theory}, 49(4):866--872, 2003.

\bibitem{MR3849557}
L.~Xu and H.~Chen.
\newblock New constant-dimension subspace codes from maximum rank distance
  codes.
\newblock {\em IEEE Trans. Inform. Theory}, 64(9):6315--6319, 2018.

\end{thebibliography}
\end{document}